\documentclass[11pt]{amsart}
\usepackage{amsfonts}

\newtheorem{theorem}{Theorem}[section]
\theoremstyle{plain}

\newtheorem{definition}{Definition}[section]

\newtheorem{lemma}{Lemma}[section]

\newtheorem{proposition}{Proposition}[section]
\newtheorem{remark}{Remark}[section]

\numberwithin{equation}{section}
\input{tcilatex}

\begin{document}
\title[Unsteady Stokes type equations]{Two-Scale convergence of \ unsteady
Stokes type equations}
\author{Lazarus SIGNING}
\dedicatory{University of Ngaoundere \\
Department of Mathematics and \\
Computer Science, P.O.Box 454\\
Ngaoundere, Cameroon\\
email: lsigning@uy1.uninet.cm}

\begin{abstract}
In this paper we study the homogenization of unsteady Stokes type equations
in the periodic setting. The usual Laplace operator involved in the
classical Stokes equations is here replaced by a linear elliptic
differential operator of divergence form with periodically oscillating
coefficients. Our mean tool is the well known two-scale convergence method.
\end{abstract}

\subjclass[2000]{ 35B27, 35B40, 73B27, 76D30}
\keywords{Homogenization, Two-scale convergence, Unsteady Stokes type
equations.}
\maketitle

\section{Introduction}

Let $\Omega $ be a smooth bounded open set in $\mathbb{R}_{x}^{N}$ (the $N$%
-numerical space $\mathbb{R}^{N}$ of variables $x=\left(
x_{1},...,x_{N}\right) $), where $N$ is a given positive integer, and let $T$
and $\varepsilon $ be real numbers with $T>0$ and $0<\varepsilon <1$. We
consider the partial differential operator 
\begin{equation*}
P^{\varepsilon }=-\sum_{i,j=1}^{N}\frac{\partial }{\partial x_{i}}\left(
a_{ij}^{\varepsilon }\frac{\partial }{\partial x_{j}}\right)
\end{equation*}%
in $\Omega \times ]0,T[$, where $a_{ij}^{\varepsilon }\left( x\right)
=a_{ij}\left( \frac{x}{\varepsilon }\right) $\quad ($x\in \Omega $), $%
a_{ij}\in L^{\infty }\left( \mathbb{R}_{y}^{N};\mathbb{R}\right) $ ($1\leq
i,j\leq N$) with 
\begin{equation}
a_{ij}=a_{ji}\text{,}  \label{eq1.1}
\end{equation}%
and the assumption that there is a constant $\alpha >0$ such that 
\begin{equation}
\sum_{i,j=1}^{N}a_{ij}\left( y\right) \zeta _{j}\zeta _{i}\geq \alpha
\left\vert \zeta \right\vert ^{2}\text{\ for all }\zeta =\left( \zeta
_{j}\right) \in \mathbb{R}^{N}\text{ and}  \label{eq1.2}
\end{equation}%
for almost all $y\in \mathbb{R}^{N}$, where $\mathbb{R}_{y}^{N}$ is the $N$%
-numerical space $\mathbb{R}^{N}$ of variables $y=\left(
y_{1},...,y_{N}\right) $, and where $\left\vert {\small \cdot }\right\vert $
denotes the Euclidean norm in $\mathbb{R}^{N}$. The operator $P^{\varepsilon
}$ acts on scalar functions, say $\varphi \in L^{2}\left( 0,T;H^{1}\left(
\Omega \right) \right) $. However, we may as well view $P^{\varepsilon }$ as
acting on vector functions $\mathbf{u}=\left( u^{i}\right) \in L^{2}\left(
0,T;H^{1}\left( \Omega \right) ^{N}\right) $ in a \textit{diagonal way},
i.e.,%
\begin{equation*}
\left( P^{\varepsilon }\mathbf{u}\right) ^{i}=P^{\varepsilon }u^{i}\text{%
\qquad }\left( i=1,...,N\right) \text{.}
\end{equation*}%
For any Roman character such as $i$, $j$ (with $1\leq i,j\leq N$), $u^{i}$
(resp. $u^{j}$) denotes the $i$-th (resp. $j$-th) component of a vector
function $\mathbf{u}$ in $L_{loc}^{1}\left( \Omega \times ]0,T[\right) ^{N}$
or in $L_{loc}^{1}\left( \mathbb{R}_{y}^{N}\times \mathbb{R}_{\tau }\right)
^{N}$ where $\mathbb{R}_{\tau }$ is the numerical space $\mathbb{R}$ of
variables $\tau $. Further, for any real $0<\varepsilon <1$, we define $%
u^{\varepsilon }$ as 
\begin{equation*}
u^{\varepsilon }\left( x,t\right) =u\left( \frac{x}{\varepsilon },\frac{t}{%
\varepsilon }\right) \text{\qquad }\left( \left( x,t\right) \in \Omega
\times ]0,T[\right)
\end{equation*}%
for $u\in L_{loc}^{1}\left( \mathbb{R}_{y}^{N}\times \mathbb{R}_{\tau
}\right) $, as is customary in homogenization theory. More generally, for $%
u\in L_{loc}^{1}\left( Q\times \mathbb{R}_{y}^{N}\times \mathbb{R}_{\tau
}\right) $ with $Q=\Omega \times ]0,T[$, it is customary to put%
\begin{equation*}
u^{\varepsilon }\left( x,t\right) =u\left( x,t,\frac{x}{\varepsilon },\frac{t%
}{\varepsilon }\right) \text{\qquad }\left( \left( x,t\right) \in \Omega
\times ]0,T[\right)
\end{equation*}%
whenever the right-hand side makes sense (see, e.g., \cite[9]{bib8}).

After these preliminaries, let $\mathbf{f}=\left( f^{i}\right) \in
L^{2}\left( 0,T;H^{-1}\left( \Omega ;\mathbb{R}\right) ^{N}\right) $. For
any fixed $0<\varepsilon <1$, we consider the initial boundary value problem%
\begin{equation}
\frac{\partial \mathbf{u}_{\varepsilon }}{\partial t}+P^{\varepsilon }%
\mathbf{u}_{\varepsilon }+\mathbf{grad}p_{\varepsilon }=\mathbf{f}\text{ in }%
\Omega \times ]0,T[\text{,}  \label{eq1.3}
\end{equation}%
\begin{equation}
div\mathbf{u}_{\varepsilon }=0\text{ in }\Omega \times ]0,T[\text{,}
\label{eq1.4}
\end{equation}%
\begin{equation}
\mathbf{u}_{\varepsilon }=0\text{ on }\partial \Omega \times ]0,T[\text{,}
\label{eq1.5}
\end{equation}%
\begin{equation}
\mathbf{u}_{\varepsilon }\left( 0\right) =0\text{ in }\Omega \text{.}
\label{eq1.6}
\end{equation}%
We will later see that as in \cite{bib18}, (\ref{eq1.3})-(\ref{eq1.6})
uniquely define $\left( \mathbf{u}_{\varepsilon },p_{\varepsilon }\right) $
with $\mathbf{u}_{\varepsilon }\in \mathcal{W}\left( 0,T\right) $ and $%
p_{\varepsilon }\in L^{2}\left( 0,T;L^{2}\left( \Omega ;\mathbb{R}\right) 
\mathfrak{/}\mathbb{R}\right) $, where 
\begin{equation*}
\mathcal{W}\left( 0,T\right) =\left\{ \mathbf{u}\in L^{2}\left( 0,T;V\right)
:\mathbf{u}^{\prime }\in L^{2}\left( 0,T;V^{\prime }\right) \right\}
\end{equation*}%
$V$ being the space of functions $\mathbf{u}$ in $H_{0}^{1}\left( \Omega ;%
\mathbb{R}\right) ^{N}$ with $div\mathbf{u}=0$ ($V^{\prime }$ is the
topological dual of $V$) and where 
\begin{equation*}
L^{2}\left( \Omega ;\mathbb{R}\right) \mathfrak{/}\mathbb{R}=\left\{ v\in
L^{2}\left( \Omega ;\mathbb{R}\right) :\int_{\Omega }vdx=0\right\} \text{.}
\end{equation*}%
Let us recall that $\mathcal{W}\left( 0,T\right) $ is provided with the norm%
\begin{equation*}
\left\Vert \mathbf{u}\right\Vert _{\mathcal{W}\left( 0,T\right) }=\left(
\left\Vert \mathbf{u}\right\Vert _{L^{2}\left( 0,T;V\right) }^{2}+\left\Vert 
\mathbf{u}^{\prime }\right\Vert _{L^{2}\left( 0,T;V^{\prime }\right)
}^{2}\right) ^{\frac{1}{2}}\text{\qquad }\left( \mathbf{u}\in \mathcal{W}%
\left( 0,T\right) \right) \text{,}
\end{equation*}%
which makes it a Hilbert space with the following properties (see \cite%
{bib18}): $\mathcal{W}\left( 0,T\right) $ is continuously embedded in $%
\mathcal{C}\left( \left[ 0,T\right] ;L^{2}\left( \Omega \right) ^{N}\right) $
and is compactly embedded in $L^{2}\left( 0,T;L^{2}\left( \Omega \right)
^{N}\right) $.

Our aim here is to investigate the asymptotic behavior, as $\varepsilon
\rightarrow 0$, of $\left( \mathbf{u}_{\varepsilon },p_{\varepsilon }\right) 
$ under the assumption that the functions $a_{ij}$ $\left( 1\leq i,j\leq
N\right) $ are periodic in the space variable $y$. The steady state version
of this problem (i.e., the homogenization of stationary Stokes type
equations) was first investigated by Bensoussan, Lions and Papanicalaou \cite%
{bib2}. These authors use the well-known approach of asymptotic expansions
combined with Tartar's energy method. We also mention the paper of Nguetseng
and the author \cite{bib16}, on the sigma-convergence of stationary
Navier-Stokes type equations.

The present study deals with the periodic homogenization of an evolution
problem for Stokes type equations.

This study is motivated by the fact that the homogenization of (\ref{eq1.3}%
)-(\ref{eq1.6}) is connected with the modelling of heterogeneous fluid flows
(see, e.g., \cite{bib19} for more details about such models).

Our approach is the well-known two-scale convergence method.

Unless otherwise specified, vector spaces throughout are considered over the
complex field, $\mathbb{C}$, and scalar functions are assumed to take
complex values. Let us recall some basic notation. If $X$ and $F$ denote a
locally compact space and a Banach space, respectively, then we write $%
\mathcal{C}\left( X;F\right) $ for continuous mappings of $X$ into $F$, and $%
\mathcal{B}\left( X;F\right) $ for those mappings in $\mathcal{C}\left(
X;F\right) $ that are bounded. We shall assume $\mathcal{B}\left( X;F\right) 
$ to be equipped with the supremum norm $\left\Vert u\right\Vert _{\infty
}=\sup_{x\in X}\left\Vert u\left( x\right) \right\Vert $ ($\left\Vert 
{\small \cdot }\right\Vert $ denotes the norm in $F$). For shortness we will
write $\mathcal{C}\left( X\right) =\mathcal{C}\left( X;\mathbb{C}\right) $
and $\mathcal{B}\left( X\right) =\mathcal{B}\left( X;\mathbb{C}\right) $.
Likewise in the case when $F=\mathbb{C}$, the usual spaces $L^{p}\left(
X;F\right) $ and $L_{loc}^{p}\left( X;F\right) $ ($X$ provided with a
positive Radon measure) will be denoted by $L^{p}\left( X\right) $ and $%
L_{loc}^{p}\left( X\right) $, respectively. Finally, the numerical space $%
\mathbb{R}^{N}$ and its open sets are each provided with Lebesgue measure
denoted by $dx=dx_{1}...dx_{N}$.

The rest of the paper is organized as follows. Section 2 is devoted to the
preliminary results on existence and uniqueness, and some estimates on the
velocity $\mathbf{u}_{\varepsilon }$, the pressure $p_{\varepsilon }$ and
the accelaration $\frac{\partial \mathbf{u}_{\varepsilon }}{\partial t}$ of
the fluid, whereas in Section 3 one convergence theorem is established.

\section{Preliminaries}

Let $\Omega $ be a smooth bounded open set in $\mathbb{R}^{N}$, let $T>0$ be
a real number and let $\mathbf{f}=\left( f^{j}\right) \in L^{2}\left(
0,T;H^{-1}\left( \Omega \right) ^{N}\right) $. For $0<\varepsilon <1$, it is
not apparent that the initial boundary value problem (\ref{eq1.3})-(\ref%
{eq1.6}) has a solution $\left( \mathbf{u}_{\varepsilon },p_{\varepsilon
}\right) $, and that the latter is unique. With a view to elucidating this,
we introduce, for fixed $0<\varepsilon <1$ the bilinear form $a^{\varepsilon
}$ on $H_{0}^{1}\left( \Omega ;\mathbb{R}\right) ^{N}\times H_{0}^{1}\left(
\Omega ;\mathbb{R}\right) ^{N}$ defined by%
\begin{equation*}
a^{\varepsilon }\left( \mathbf{u},\mathbf{v}\right)
=\sum_{k=1}^{N}\sum_{i,j=1}^{N}\int_{\Omega }a_{ij}^{\varepsilon }\frac{%
\partial u^{k}}{\partial x_{j}}\frac{\partial v^{k}}{\partial x_{i}}dx
\end{equation*}%
for $\mathbf{u}=\left( u^{k}\right) $ and $\mathbf{v}=\left( v^{k}\right) $
in $H_{0}^{1}\left( \Omega ;\mathbb{R}\right) ^{N}$. By virtue of (\ref%
{eq1.1}), the form $a^{\varepsilon }$ is symmetric. Further, in view of (\ref%
{eq1.2}), 
\begin{equation}
a^{\varepsilon }\left( \mathbf{v},\mathbf{v}\right) \geq \alpha \left\Vert 
\mathbf{v}\right\Vert _{H_{0}^{1}\left( \Omega \right) ^{N}}^{2}
\label{eq2.1}
\end{equation}%
for every $\mathbf{v}=\left( v^{k}\right) \in H_{0}^{1}\left( \Omega ;%
\mathbb{R}\right) ^{N}$ and $0<\varepsilon <1$, where 
\begin{equation*}
\left\Vert \mathbf{v}\right\Vert _{H_{0}^{1}\left( \Omega \right)
^{N}}=\left( \sum_{k=1}^{N}\int_{\Omega }\left\vert \nabla v^{k}\right\vert
dx\right) ^{\frac{1}{2}}
\end{equation*}%
with $\nabla v^{k}=\left( \frac{\partial v^{k}}{\partial x_{1}},...,\frac{%
\partial v^{k}}{\partial x_{N}}\right) $. On the other hand, it is clear
that a constant $c_{0}>0$ exists such that%
\begin{equation}
\left\vert a^{\varepsilon }\left( \mathbf{u},\mathbf{v}\right) \right\vert
\leq c_{0}\left\Vert \mathbf{u}\right\Vert _{H_{0}^{1}\left( \Omega \right)
^{N}}\left\Vert \mathbf{v}\right\Vert _{H_{0}^{1}\left( \Omega \right) ^{N}}
\label{eq2.2}
\end{equation}%
for every $\mathbf{u}$, $\mathbf{v}\in H_{0}^{1}\left( \Omega ;\mathbb{R}%
\right) ^{N}$ and $0<\varepsilon <1$.

We are now in a position to verify the following result.

\begin{proposition}
\label{pr2.1} Suppose $\mathbf{f}$ lies in $L^{2}\left( 0,T;L^{2}\left(
\Omega ;\mathbb{R}\right) ^{N}\right) $. Under the hypotheses (\ref{eq1.1})-(%
\ref{eq1.2}), the initial boundary value problem (\ref{eq1.3})-(\ref{eq1.6})
determines a unique pair $\left( \mathbf{u}_{\varepsilon },p_{\varepsilon
}\right) $ with

$\mathbf{u}_{\varepsilon }\in L^{2}\left( 0,T;H_{0}^{1}\left( \Omega ;%
\mathbb{R}\right) ^{N}\right) \cap \mathcal{C}\left( \left[ 0,T\right]
;L^{2}\left( \Omega ;\mathbb{R}\right) ^{N}\right) $ and

$p_{\varepsilon }\in L^{2}\left( 0,T;L^{2}\left( \Omega ;\mathbb{R}\right) 
\mathfrak{/}\mathbb{R}\right) $.
\end{proposition}

\begin{proof}
For fixed $0<\varepsilon <1$, we consider the Cauchy problem%
\begin{equation}
\left\{ 
\begin{array}{c}
\mathbf{u}_{\varepsilon }^{\prime }\left( t\right) +A_{\varepsilon }\mathbf{u%
}_{\varepsilon }\left( t\right) =\mathbf{\ell }\left( t\right) \text{ in }%
]0,T[ \\ 
\mathbf{u}_{\varepsilon }\left( 0\right) =0\text{,}\qquad \qquad \qquad
\qquad \qquad%
\end{array}%
\right.  \label{eq2.3}
\end{equation}%
where $A_{\varepsilon }$ is the linear operator of $V$ into $V^{\prime }$
defined by%
\begin{equation*}
\left( A_{\varepsilon }\mathbf{u},\mathbf{v}\right) =a^{\varepsilon }\left( 
\mathbf{u},\mathbf{v}\right) \text{ for all }\mathbf{u},\mathbf{v}\in V
\end{equation*}%
and $\mathbf{\ell }$ is the function in $L^{2}\left( 0,T;V^{\prime }\right) $
defined by%
\begin{equation*}
\left( \mathbf{\ell }\left( t\right) ,\mathbf{v}\right) =\left( \mathbf{f}%
\left( t\right) ,\mathbf{v}\right) \text{ for all }\mathbf{v}\in V
\end{equation*}%
and for almost all $t\in ]0,T[$, and where $\left( ,\right) $ denotes the
duality pairing between $V^{\prime }$ and $V$ as well as between $%
H^{-1}\left( \Omega ;\mathbb{R}\right) ^{N}$ and $H_{0}^{1}\left( \Omega ;%
\mathbb{R}\right) ^{N}$. Thanks to (\ref{eq2.1})-(\ref{eq2.2}) the Cauchy
problem (\ref{eq2.3}) admits a unique solution $\mathbf{u}_{\varepsilon }$
in $\mathcal{W}\left( 0,T\right) $, as is easily seen by following \cite[%
Chap.3, Th\'{e}or\`{e}me 1.2, p.116]{bib5}, see also \cite[pp.254-260]{bib18}%
. Now, let us check that the abstract parabolic problem (\ref{eq2.3}) is
equivalent to (\ref{eq1.3})-(\ref{eq1.6}). Let $\mathbf{U}_{\varepsilon
}\left( t\right) =\int_{0}^{t}P^{\varepsilon }\mathbf{u}_{\varepsilon
}\left( s\right) ds$ and $\mathbf{F}\left( t\right) =\int_{0}^{t}\mathbf{f}%
\left( s\right) ds$ for $0\leq t\leq T$, where $\mathbf{u}_{\varepsilon }$
satisfies (\ref{eq2.3}). It is evident that $\mathbf{U}_{\varepsilon }$ and $%
\mathbf{F}\in \mathcal{C}\left( \left[ 0,T\right] ;H^{-1}\left( \Omega ;%
\mathbb{R}\right) ^{N}\right) $. By the first equality of (\ref{eq2.3}) we
have 
\begin{equation}
\frac{d}{dt}\left( \mathbf{u}_{\varepsilon }\left( t\right) ,\mathbf{\varphi 
}\right) =\left( \mathbf{\ell }\left( t\right) -A_{\varepsilon }\mathbf{u}%
_{\varepsilon }\left( t\right) ,\mathbf{\varphi }\right) \text{ for all }%
\mathbf{\varphi }\in \mathcal{V}\text{, }  \label{eq2.4}
\end{equation}%
where 
\begin{equation*}
\mathcal{V}=\left\{ \mathbf{\varphi }\in \mathcal{D}\left( \Omega ;\mathbb{R}%
\right) ^{N}:div\mathbf{\varphi }=0\right\} \text{.}
\end{equation*}%
Integrating (\ref{eq2.4}), we have%
\begin{equation*}
\left\langle \mathbf{u}_{\varepsilon }\left( t\right) +\mathbf{U}%
_{\varepsilon }\left( t\right) -\mathbf{F}\left( t\right) ,\mathbf{\varphi }%
\right\rangle =0\text{, }0\leq t\leq T\text{, }\mathbf{\varphi }\in \mathcal{%
V}\text{.}
\end{equation*}%
Thus, using a classical argument (see, e.g., \cite[p.14]{bib18}), we get a
function $P_{\varepsilon }\in \mathcal{C}\left( \left[ 0,T\right]
;L^{2}\left( \Omega ;\mathbb{R}\right) /\mathbb{R}\right) $ such that%
\begin{equation*}
\mathbf{u}_{\varepsilon }+\mathbf{U}_{\varepsilon }+\mathbf{grad}%
P_{\varepsilon }=\mathbf{F}\text{.}
\end{equation*}%
Hence $p_{\varepsilon }=\frac{\partial P_{\varepsilon }}{\partial t}\in 
\mathcal{D}^{\prime }\left( Q\right) $ and the pair $\left( \mathbf{u}%
_{\varepsilon },p_{\varepsilon }\right) $ verifies (\ref{eq1.3}) (in the
distribution sense on $Q$), with in addition (\ref{eq1.4})-(\ref{eq1.6}), of
course. Futhermore, by using the fact that $\mathbf{f}\in L^{2}\left(
0,T;L^{2}\left( \Omega ;\mathbb{R}\right) ^{N}\right) $ we have $\mathbf{u}%
_{\varepsilon }^{\prime }\in L^{2}\left( 0,T;L^{2}\left( \Omega ;\mathbb{R}%
\right) ^{N}\right) $, as is easily seen by following \cite[p.268]{bib18}.
Therefore $p_{\varepsilon }$ lies in $L^{2}\left( 0,T;L^{2}\left( \Omega ;%
\mathbb{R}\right) \mathfrak{/}\mathbb{R}\right) $ and is unique. Conversely,
it is an easy exercise to verify that if $\left( \mathbf{u}_{\varepsilon
},p_{\varepsilon }\right) $ is a solution of (\ref{eq1.3})-(\ref{eq1.6})
with $\mathbf{u}_{\varepsilon }\in \mathcal{W}\left( 0,T\right) $ and $%
p_{\varepsilon }\in L^{2}\left( 0,T;L^{2}\left( \Omega ;\mathbb{R}\right)
\right) $, then $\mathbf{u}_{\varepsilon }$ satisfies (\ref{eq2.3}). The
proof is complete.
\end{proof}

The following regularity result is fundamental for the estimates of the
solution $\left( \mathbf{u}_{\varepsilon },p_{\varepsilon }\right) $ of (\ref%
{eq1.3})-(\ref{eq1.6}).

\begin{lemma}
\label{lem2.1} Suppose $\mathbf{f},\mathbf{f}^{\prime }\in L^{2}\left(
0,T;V^{\prime }\right) $ and $\mathbf{f}\left( 0\right) \in L^{2}\left(
\Omega ;\mathbb{R}\right) ^{N}$. Then the solution $\mathbf{u}_{\varepsilon
} $ of (\ref{eq2.3}) verifies:%
\begin{equation*}
\mathbf{u}_{\varepsilon }^{\prime }\in L^{2}\left( 0,T;V\right) \cap
L^{\infty }\left( 0,T;H\right) \text{,}
\end{equation*}%
where $H$ is the closure of $\mathcal{V}$ in $L^{2}\left( \Omega ;\mathbb{R}%
\right) ^{N}$.
\end{lemma}

The proof of the above lemma follows by the same line of argument as in the
proof of \cite[p.299, Theorem 3.5]{bib18}. So we omit it. We are now able to
prove the result on the estimates.

\begin{proposition}
\label{pr2.2} Under the hypotheses of Lemma \ref{lem2.1}, there exists a
constant $c>0$ (independent of $\varepsilon $) such that the pair $\left( 
\mathbf{u}_{\varepsilon },p_{\varepsilon }\right) $ solution of (\ref{eq1.3}%
)-(\ref{eq1.6}) in $\mathcal{W}\left( 0,T\right) \times L^{2}\left(
0,T;L^{2}\left( \Omega ;\mathbb{R}\right) \mathfrak{/}\mathbb{R}\right) $
satisfies:%
\begin{equation}
\left\Vert \mathbf{u}_{\varepsilon }\right\Vert _{\mathcal{W}\left(
0,T\right) }\leq c  \label{eq2.5}
\end{equation}%
\begin{equation}
\left\Vert \frac{\partial \mathbf{u}_{\varepsilon }}{\partial t}\right\Vert
_{L^{2}\left( 0,T;H^{-1}\left( \Omega \right) ^{N}\right) }\leq c
\label{eq2.6}
\end{equation}%
and 
\begin{equation}
\left\Vert p_{\varepsilon }\right\Vert _{L^{2}\left( 0,T;L^{2}\left( \Omega
\right) \right) }\leq c\text{.}  \label{eq2.7}
\end{equation}
\end{proposition}

\begin{proof}
Let $\left( \mathbf{u}_{\varepsilon },p_{\varepsilon }\right) $ be the
solution of (\ref{eq1.3})-(\ref{eq1.6}). We have%
\begin{equation}
\left( \mathbf{u}_{\varepsilon }^{\prime }\left( t\right) ,\mathbf{v}\right)
+a^{\varepsilon }\left( \mathbf{u}_{\varepsilon }\left( t\right) ,\mathbf{v}%
\right) =\left( \mathbf{f}\left( t\right) ,\mathbf{v}\right) \text{\quad }%
\left( \mathbf{v}\in V\right)  \label{eq2.8}
\end{equation}%
for almost all $t\in \left[ 0,T\right] $. By taking in particular $\mathbf{v}%
=\mathbf{u}_{\varepsilon }\left( t\right) $ in (\ref{eq2.8}), we have for
almost all $t\in \left[ 0,T\right] $ 
\begin{equation*}
\frac{d}{dt}\left\vert \mathbf{u}_{\varepsilon }\left( t\right) \right\vert
^{2}+2\alpha \left\Vert \mathbf{u}_{\varepsilon }\left( t\right) \right\Vert
^{2}\leq \frac{1}{\alpha }\left\Vert \mathbf{f}\left( t\right) \right\Vert
_{V^{\prime }}^{2}+\alpha \left\Vert \mathbf{u}_{\varepsilon }\left(
t\right) \right\Vert ^{2}
\end{equation*}%
where $\left\vert {\small \cdot }\right\vert $ and $\left\Vert {\small \cdot 
}\right\Vert $ are respectively the norms in $L^{2}\left( \Omega \right)
^{N} $ and $H_{0}^{1}\left( \Omega \right) ^{N}$. Hence, for every $s\in %
\left[ 0,T\right] $ 
\begin{equation*}
\left\vert \mathbf{u}_{\varepsilon }\left( s\right) \right\vert ^{2}+\alpha
\int_{0}^{s}\left\Vert \mathbf{u}_{\varepsilon }\left( t\right) \right\Vert
^{2}dt\leq \frac{1}{\alpha }\int_{0}^{T}\left\Vert \mathbf{f}\left( t\right)
\right\Vert _{V^{\prime }}^{2}dt
\end{equation*}%
since $\mathbf{u}_{\varepsilon }\left( 0\right) =0$. By the preceding
inequality, we see that%
\begin{equation}
\alpha \int_{0}^{T}\left\Vert \mathbf{u}_{\varepsilon }\left( t\right)
\right\Vert ^{2}dt\leq \frac{1}{\alpha }\int_{0}^{T}\left\Vert \mathbf{f}%
\left( t\right) \right\Vert _{V^{\prime }}^{2}dt\text{.}  \label{eq2.9}
\end{equation}%
On the other hand, the abstract parabolic problem (\ref{eq2.3}) gives%
\begin{equation*}
\mathbf{u}_{\varepsilon }^{\prime }=\mathbf{f}-A_{\varepsilon }\mathbf{u}%
_{\varepsilon }\text{.}
\end{equation*}%
Hence, in view of (\ref{eq2.2})%
\begin{equation}
\left\Vert \mathbf{u}_{\varepsilon }^{\prime }\right\Vert _{L^{2}\left(
0,T;V^{\prime }\right) }\leq \left\Vert \mathbf{f}\right\Vert _{L^{2}\left(
0,T;V^{\prime }\right) }+c_{0}\left\Vert \mathbf{u}_{\varepsilon
}\right\Vert _{L^{2}\left( 0,T;V\right) }\text{.}  \label{eq2.10}
\end{equation}%
Thus, by (\ref{eq2.9}) and (\ref{eq2.10}) one quickly arrives at (\ref{eq2.5}%
). Let us show (\ref{eq2.6}). We are allowed to differentiate (\ref{eq2.8})
in distribution sense on $]0,T[$, and by virtue of the hypotheses of Lemma %
\ref{lem2.1}, we get $\mathbf{u}_{\varepsilon }^{\prime \prime }\in
L^{2}\left( 0,T;V^{\prime }\right) $ and 
\begin{equation}
\left( \mathbf{u}_{\varepsilon }^{\prime \prime },\mathbf{v}\right)
+a^{\varepsilon }\left( \mathbf{u}_{\varepsilon }^{\prime },\mathbf{v}%
\right) =\left( \mathbf{f}^{\prime },\mathbf{v}\right) \text{\quad }\left( 
\mathbf{v}\in V\right) \text{.}  \label{eq2.10a}
\end{equation}%
In view of Lemma \ref{lem2.1}, we take in particular $\mathbf{v}=\mathbf{u}%
_{\varepsilon }^{\prime }\left( t\right) $ in (\ref{eq2.8}). This yields 
\begin{equation*}
\left\vert \mathbf{u}_{\varepsilon }^{\prime }\left( t\right) \right\vert
^{2}+a^{\varepsilon }\left( \mathbf{u}_{\varepsilon }\left( t\right) ,%
\mathbf{u}_{\varepsilon }^{\prime }\left( t\right) \right) =\left( \mathbf{f}%
\left( t\right) ,\mathbf{u}_{\varepsilon }^{\prime }\left( t\right) \right)
\quad \left( t\in \left[ 0,T\right] \right) \text{.}
\end{equation*}%
Further, since $\mathbf{u}_{\varepsilon }^{\prime }\in L^{2}\left(
0,T;V\right) $ and $\mathbf{u}_{\varepsilon }^{\prime \prime }\in
L^{2}\left( 0,T;V^{\prime }\right) $, we have $\mathbf{u}_{\varepsilon
}^{\prime }\in \mathcal{C}\left( \left[ 0,T\right] ;H\right) $. Hence, by
taking in particular $t=0$ in the preceding equality and using (\ref{eq1.6})
one quickly arrives at 
\begin{equation*}
\left\vert \mathbf{u}_{\varepsilon }^{\prime }\left( 0\right) \right\vert
^{2}\leq \left\vert \mathbf{f}\left( 0\right) \right\vert \left\vert \mathbf{%
u}_{\varepsilon }^{\prime }\left( 0\right) \right\vert \text{,}
\end{equation*}%
i.e.,%
\begin{equation}
\left\vert \mathbf{u}_{\varepsilon }^{\prime }\left( 0\right) \right\vert
\leq \left\vert \mathbf{f}\left( 0\right) \right\vert \text{.}
\label{eq2.11}
\end{equation}%
The inequality (\ref{eq2.11}) shows that $\mathbf{u}_{\varepsilon }^{\prime
}\left( 0\right) $ lies in a bounded subset of $H$.\ On\ the other hand, by
taking in particular $\mathbf{v}=\mathbf{u}_{\varepsilon }^{\prime }\left(
t\right) $ in (\ref{eq2.10a}), we get%
\begin{equation*}
\frac{d}{dt}\left\vert \mathbf{u}_{\varepsilon }^{\prime }\left( t\right)
\right\vert ^{2}+2\alpha \left\Vert \mathbf{u}_{\varepsilon }^{\prime
}\left( t\right) \right\Vert ^{2}\leq \frac{1}{\alpha }\left\Vert \mathbf{f}%
^{\prime }\left( t\right) \right\Vert _{V^{\prime }}^{2}+\alpha \left\Vert 
\mathbf{u}_{\varepsilon }^{\prime }\left( t\right) \right\Vert ^{2}
\end{equation*}%
for almost all $t\in \left[ 0,T\right] $. Integrating the preceding
inequality on $\left[ 0,t\right] $ ($t\in \left[ 0,T\right] $) leads to 
\begin{equation*}
\left\vert \mathbf{u}_{\varepsilon }^{\prime }\left( t\right) \right\vert
^{2}+\alpha \int_{0}^{t}\left\Vert \mathbf{u}_{\varepsilon }^{\prime }\left(
s\right) \right\Vert ^{2}ds\leq \frac{1}{\alpha }\left\Vert \mathbf{f}%
^{\prime }\right\Vert _{L^{2}\left( 0,T;V^{\prime }\right) }^{2}+\left\vert 
\mathbf{u}_{\varepsilon }^{\prime }\left( 0\right) \right\vert ^{2}\text{.}
\end{equation*}%
It follows from (\ref{eq2.11}) and the preceding inequality that $\mathbf{u}%
_{\varepsilon }^{\prime }$ belongs to a bounded subset of $L^{2}\left(
0,T;L^{2}\left( \Omega ;\mathbb{R}\right) ^{N}\right) $. Hence (\ref{eq2.6})
is immediate. Let us prove (\ref{eq2.7}). For almost all $t\in \left[ 0,T%
\right] $, $p_{\varepsilon }\left( t\right) \in L^{2}\left( \Omega ;\mathbb{R%
}\right) \mathfrak{/}\mathbb{R}$. Thus, by \cite[p. 30]{bib17} there exists $%
\mathbf{v}_{\varepsilon }\left( t\right) \in $ $H_{0}^{1}\left( \Omega ;%
\mathbb{R}\right) ^{N}$ such that%
\begin{equation}
div\mathbf{v}_{\varepsilon }\left( t\right) =p_{\varepsilon }\left( t\right)
\label{eq2.12a}
\end{equation}%
\begin{equation}
\left\Vert \mathbf{v}_{\varepsilon }\left( t\right) \right\Vert \leq
c_{1}\left\vert p_{\varepsilon }\left( t\right) \right\vert _{L^{2}\left(
\Omega \right) }\text{,}  \label{eq2.12b}
\end{equation}%
where the constant $c_{1}$ depends solely on $\Omega $. Multiplying (\ref%
{eq1.3}) by $\mathbf{v}_{\varepsilon }\left( t\right) $ yields 
\begin{equation*}
\left( \mathbf{u}_{\varepsilon }^{\prime }\left( t\right) ,\mathbf{v}%
_{\varepsilon }\left( t\right) \right) +a^{\varepsilon }\left( \mathbf{u}%
_{\varepsilon }\left( t\right) ,\mathbf{v}_{\varepsilon }\left( t\right)
\right) -\int_{\Omega }p_{\varepsilon }\left( t\right) div\mathbf{v}%
_{\varepsilon }\left( t\right) dx=\left( \mathbf{f}\left( t\right) ,\mathbf{v%
}_{\varepsilon }\left( t\right) \right)
\end{equation*}%
for almost all $t\in \left[ 0,T\right] $. Integrating the preceding equality
on $\left[ 0,T\right] $ and using (\ref{eq2.12a})-(\ref{eq2.12b}) lead to 
\begin{equation*}
\left\Vert p_{\varepsilon }\right\Vert _{L^{2}\left( Q\right) }^{2}\leq
c_{1}c\left\Vert \mathbf{u}_{\varepsilon }^{\prime }\right\Vert
_{L^{2}\left( 0,T;L^{2}\left( \Omega \right) ^{N}\right) }\left\Vert
p_{\varepsilon }\right\Vert _{L^{2}\left( Q\right) }+c_{1}\left\Vert \mathbf{%
f}\right\Vert _{L^{2}\left( 0,T;H^{-1}\left( \Omega \right) \right)
}\left\Vert p_{\varepsilon }\right\Vert _{L^{2}\left( Q\right) }
\end{equation*}%
\begin{equation*}
+c_{1}c_{0}\left\Vert \mathbf{u}_{\varepsilon }\right\Vert _{L^{2}\left(
0,T;V\right) }\left\Vert p_{\varepsilon }\right\Vert _{L^{2}\left( Q\right) }%
\text{,}
\end{equation*}%
where $c$ is the constant in the Poincar\'{e} inequality, $c_{0}$ and $c_{1}$
are the constants in (\ref{eq2.2}) and (\ref{eq2.12b})\ respectively. Thus, 
\begin{equation}
\left\Vert p_{\varepsilon }\right\Vert _{L^{2}\left( Q\right) }\leq
c_{1}c\left\Vert \mathbf{u}_{\varepsilon }^{\prime }\right\Vert
_{L^{2}\left( 0,T;L^{2}\left( \Omega \right) ^{N}\right) }+c_{1}\left\Vert 
\mathbf{f}\right\Vert _{L^{2}\left( 0,T;H^{-1}\left( \Omega \right) \right)
}+c_{1}c_{0}\left\Vert \mathbf{u}_{\varepsilon }\right\Vert _{L^{2}\left(
0,T;V\right) }\text{.}  \label{eq2.13}
\end{equation}%
Combining (\ref{eq2.13}), (\ref{eq2.5}) and (\ref{eq2.6}) leads to (\ref%
{eq2.7}).
\end{proof}

\section{A convergence result for (\protect\ref{eq1.3})-(\protect\ref{eq1.6})%
}

We set $Y=\left( -\frac{1}{2},\frac{1}{2}\right) ^{N}$, $Y$ considered as a
subset of $\mathbb{R}_{y}^{N}$ (the space $\mathbb{R}^{N}$ of variables $%
y=\left( y_{1},...,y_{N}\right) $). We set also $Z=\left( -\frac{1}{2},\frac{%
1}{2}\right) $, $Z$ considered as a subset of $\mathbb{R}_{\tau }$ (the
space $\mathbb{R}$ of variables $\tau $). Our purpose is to study the
homogenization of (\ref{eq1.3})-(\ref{eq1.6}) under the periodicity
hypothesis on $a_{ij}$.

\subsection{\textbf{Preliminaries}}

Let us first recall that a function $u\in L_{loc}^{1}\left( \mathbb{R}%
_{y}^{N}\times \mathbb{R}_{\tau }\right) $ is said to be $Y\times Z$%
-periodic if for each $\left( k,l\right) \in \mathbb{Z}^{N}\times \mathbb{Z}$
($\mathbb{Z}$ denotes the integers), we have $u\left( y+k,\tau +l\right)
=u\left( y,\tau \right) $ almost everywhere (a.e.) in $\left( y,\tau \right)
\in \mathbb{R}^{N}\times \mathbb{R}$. If in addition $u$ is continuous, then
the preceding equality holds for every $\left( y,\tau \right) \in \mathbb{R}%
^{N}\times \mathbb{R}$, of course. The space of all $Y\times Z$-periodic
continuous complex functions on $\mathbb{R}_{y}^{N}\times \mathbb{R}_{\tau }$
is denoted by $\mathcal{C}_{per}\left( Y\times Z\right) $; that of all $%
Y\times Z$-periodic functions in $L_{loc}^{p}\left( \mathbb{R}_{y}^{N}\times 
\mathbb{R}_{\tau }\right) $ $\left( 1\leq p<\infty \right) $ is denoted by $%
L_{per}^{p}\left( Y\times Z\right) $. $\mathcal{C}_{per}\left( Y\times
Z\right) $ is a Banach space under the supremum norm on $\mathbb{R}%
^{N}\times \mathbb{R}$, whereas $L_{per}^{p}\left( Y\times Z\right) $ is a
Banach space under the norm 
\begin{equation*}
\left\Vert u\right\Vert _{L^{p}\left( Y\times Z\right) }=\left(
\int_{Z}\int_{Y}\left\vert u\left( y,\tau \right) \right\vert ^{p}dyd\tau
\right) ^{\frac{1}{p}}\text{ }\left( u\in L_{per}^{p}\left( Y\times Z\right)
\right) \text{.}
\end{equation*}

We will need the space $H_{\#}^{1}\left( Y\right) $ of $Y$-periodic
functions $u\in H_{loc}^{1}\left( \mathbb{R}_{y}^{N}\right)
=W_{loc}^{1,2}\left( \mathbb{R}_{y}^{N}\right) $ such that $\int_{Y}u\left(
y\right) dy=0$. Provided with the gradient norm, 
\begin{equation*}
\left\Vert u\right\Vert _{H_{\#}^{1}\left( Y\right) }=\left(
\int_{Y}\left\vert \nabla _{y}u\right\vert ^{2}dy\right) ^{\frac{1}{2}}\text{
}\left( u\in H_{\#}^{1}\left( Y\right) \right) \text{,}
\end{equation*}%
where $\nabla _{y}u=\left( \frac{\partial u}{\partial y_{1}},...,\frac{%
\partial u}{\partial y_{N}}\right) $, $H_{\#}^{1}\left( Y\right) $ is a
Hilbert space. We will also need the space $L_{per}^{2}\left(
Z;H_{\#}^{1}\left( Y\right) \right) $ with the norm%
\begin{equation*}
\left\Vert u\right\Vert _{L_{per}^{2}\left( Z;H_{\#}^{1}\left( Y\right)
\right) }=\left( \int_{Z}\int_{Y}\left\vert \nabla _{y}u\left( y,\tau
\right) \right\vert ^{2}dyd\tau \right) ^{\frac{1}{2}}\text{ }\left( u\in
L_{per}^{2}\left( Z;H_{\#}^{1}\left( Y\right) \right) \right)
\end{equation*}%
which is a Hilbert space.

Before we can recall the concept of two-scale convergence, let us introduce
one further notation. The letter $E$ throughout will denote a family of real
numbers $0<\varepsilon <1$ admitting $0$ as an accumulation point. For
example, $E$ may be the whole interval $\left( 0,1\right) $; $E$ may also be
an ordinary sequence $\left( \varepsilon _{n}\right) _{n\in \mathbb{N}}$
with $0<\varepsilon _{n}<1$ and $\varepsilon _{n}\rightarrow 0$ as $%
n\rightarrow \infty $. In the latter case $E$ will be referred to as a 
\textit{fundamental sequence}.

Let $\Omega $ be a bounded open set in $\mathbb{R}_{x}^{N}$ and $Q=\Omega
\times ]0,T[$ with $T\in \mathbb{R}_{+}^{\ast }$, and let $1\leq p<\infty $.

\begin{definition}
\label{def2.1} A sequence $\left( u_{\varepsilon }\right) _{\varepsilon \in
E}\subset L^{p}\left( Q\right) $ is said to:

(i) weakly two-scale converge in $L^{p}\left( Q\right) $ to some $u_{0}\in
L^{p}\left( Q;L_{per}^{p}\left( Y\times Z\right) \right) $ if as

\noindent $E\ni \varepsilon \rightarrow 0$, 
\begin{equation}
\int_{Q}u_{\varepsilon }\left( x,t\right) \psi ^{\varepsilon }\left(
x,t\right) dxdt\rightarrow \int \int \int_{Q\times Y\times Z}u_{0}\left(
x,t,y,\tau \right) \psi \left( x,t,y,\tau \right) dxdtdyd\tau  \label{eq3.1}
\end{equation}%
\begin{equation*}
\begin{array}{c}
\text{for all }\psi \in L^{p^{\prime }}\left( Q;\mathcal{C}_{per}\left(
Y\times Z\right) \right) \text{ }\left( \frac{1}{p^{\prime }}=1-\frac{1}{p}%
\right) \text{, where }\psi ^{\varepsilon }\left( x,t\right) = \\ 
\psi \left( x,t,\frac{x}{\varepsilon },\frac{t}{\varepsilon }\right) \text{ }%
\left( \left( x,t\right) \in Q\right) \text{;}%
\end{array}%
\end{equation*}

(ii) strongly two-scale converge in $L^{p}\left( Q\right) $ to some $%
u_{0}\in L^{p}\left( Q;L_{per}^{p}\left( Y\times Z\right) \right) $ if the
following property is verified: 
\begin{equation*}
\left\{ 
\begin{array}{c}
\text{Given }\eta >0\text{ and }v\in L^{p}\left( Q;\mathcal{C}_{per}\left(
Y\times Z\right) \right) \text{ with} \\ 
\left\Vert u_{0}-v\right\Vert _{L^{p}\left( Q\times Y\times Z\right) }\leq 
\frac{\eta }{2}\text{, there is some }\alpha >0\text{ such} \\ 
\text{that }\left\Vert u_{\varepsilon }-v^{\varepsilon }\right\Vert
_{L^{p}\left( Q\right) }\leq \eta \text{ provided }E\ni \varepsilon \leq
\alpha \text{.}%
\end{array}%
\right.
\end{equation*}
\end{definition}

We will briefly express weak and strong two-scale convergence by writing $%
u_{\varepsilon }\rightarrow u_{0}$ in $L^{p}\left( Q\right) $-weak $2$-$s$
and $u_{\varepsilon }\rightarrow u_{0}$ in $L^{p}\left( Q\right) $-strong $2$%
-$s$, respectively.

\begin{remark}
\label{rem2.1} It is of interest to know that if $u_{\varepsilon
}\rightarrow u_{0}$ in $L^{p}\left( Q\right) $-weak $2$-$s$, then (\ref%
{eq3.1}) holds for $\psi \in \mathcal{C}\left( \overline{Q};L_{per}^{\infty
}\left( Y\times Z\right) \right) $. See \cite[Proposition 10]{bib9} for the
proof.
\end{remark}

Instead of repeating here the main results underlying two-scale convergence,
we find it more convenient to draw the reader's attention to a few
references, see, e.g., \cite{bib1}, \cite{bib7}, \cite{bib9} and \cite{bib20}%
.

However, we recall below two fundamental results. First of all, let 
\begin{equation*}
\mathcal{Y}\left( 0,T\right) =\left\{ v\in L^{2}\left( 0,T;H_{0}^{1}\left(
\Omega ;\mathbb{R}\right) \right) :v^{\prime }\in L^{2}\left(
0,T;H^{-1}\left( \Omega ;\mathbb{R}\right) \right) \right\} \text{.}
\end{equation*}%
$\mathcal{Y}\left( 0,T\right) $ is provided with the norm 
\begin{equation*}
\left\Vert v\right\Vert _{\mathcal{Y}\left( 0,T\right) }=\left( \left\Vert
v\right\Vert _{L^{2}\left( 0,T;H_{0}^{1}\left( \Omega \right) \right)
}^{2}+\left\Vert v^{\prime }\right\Vert _{L^{2}\left( 0,T;H^{-1}\left(
\Omega \right) \right) }^{2}\right) ^{\frac{1}{2}}\qquad \left( v\in 
\mathcal{Y}\left( 0,T\right) \right)
\end{equation*}%
which makes it a Hilbert space.

\begin{theorem}
\label{th3.1} Assume that $1<p<\infty $ and further $E$ is a fundamental
sequence. Let a sequence $\left( u_{\varepsilon }\right) _{\varepsilon \in
E} $ be bounded in $L^{p}\left( Q\right) $. Then, a subsequence $E^{\prime }$
can be extracted from $E$ such that $\left( u_{\varepsilon }\right)
_{\varepsilon \in E^{\prime }}$ weakly two-scale converges in $L^{p}\left(
Q\right) $.
\end{theorem}

\begin{theorem}
\label{th3.2} Let $E$ be a fundamental sequence. Suppose a sequence $\left(
u_{\varepsilon }\right) _{\varepsilon \in E}$ is bounded in $\mathcal{Y}%
\left( 0,T\right) $. Then, a subsequence $E^{\prime }$ can be extracted from 
$E$ such that, as $E^{\prime }\ni \varepsilon \rightarrow 0$, 
\begin{equation*}
u_{\varepsilon }\rightarrow u_{0}\text{ in }\mathcal{Y}\left( 0,T\right) 
\text{-weak,\qquad \qquad \qquad \qquad \qquad \qquad }
\end{equation*}%
\begin{equation*}
u_{\varepsilon }\rightarrow u_{0}\text{ in }L^{2}\left( Q\right) \text{-weak 
}2\text{-}s\text{,\qquad \qquad \qquad \qquad \qquad }\quad
\end{equation*}%
\begin{equation*}
\frac{\partial u_{\varepsilon }}{\partial x_{j}}\rightarrow \frac{\partial
u_{0}}{\partial x_{j}}+\frac{\partial u_{1}}{\partial y_{j}}\text{ in }%
L^{2}\left( Q\right) \text{-weak }2\text{-}s\text{ }\left( 1\leq j\leq
N\right) \text{,}
\end{equation*}%
where $u_{0}\in \mathcal{Y}\left( 0,T\right) $, $u_{1}\in L^{2}\left(
Q;L_{per}^{2}\left( Z;H_{\#}^{1}\left( Y\right) \right) \right) $.
\end{theorem}

The proof of Theorem \ref{th3.1} can be found in, e.g., \cite{bib1}, \cite%
{bib7}, whereas Theorem \ref{th3.2} has its proof in, e.g., \cite{bib9} and 
\cite{bib15}.

\subsection{\textbf{A global homogenization theorem}}

Before we can establish a so-called global homogenization theorem for (\ref%
{eq1.3})-(\ref{eq1.6}), we require a few basic notation and results. To
begin, let 
\begin{equation*}
\mathcal{V}_{Y}=\left\{ \mathbf{\psi }\in \mathcal{C}_{per}^{\infty }\left(
Y;\mathbb{R}\right) ^{N}:\int_{Y}\mathbf{\psi }\left( y\right) dy=0,\text{ }%
div_{y}\mathbf{\psi =}0\right\} \text{, }
\end{equation*}%
\begin{equation*}
V_{Y}=\left\{ \mathbf{w}\in H_{\#}^{1}\left( Y;\mathbb{R}\right) ^{N}:div_{y}%
\mathbf{w=}0\right\} \text{, }
\end{equation*}%
where: $\mathcal{C}_{per}^{\infty }\left( Y;\mathbb{R}\right) =\mathcal{C}%
^{\infty }\left( \mathbb{R}^{N};\mathbb{R}\right) \cap \mathcal{C}%
_{per}\left( Y\right) $, $div_{y}$ denotes the divergence operator in $%
\mathbb{R}_{y}^{N}$. We provide $V_{Y}$ with the $H_{\#}^{1}\left( Y\right)
^{N}$-norm, which makes it a Hilbert space. There is no difficulty in
verifying that $\mathcal{V}_{Y}$ is dense in $V_{Y}$ (proceed as in \cite[%
Proposition 3.2]{bib14}). With this in mind, set 
\begin{equation*}
\mathbb{F}_{0}^{1}=L^{2}\left( 0,T;V\right) \times L^{2}\left(
Q;L_{per}^{2}\left( Z;V_{Y}\right) \right) \text{.}
\end{equation*}%
This is a Hilbert space with norm 
\begin{equation*}
\left\Vert \mathbf{v}\right\Vert _{\mathbb{F}_{0}^{1}}=\left( \left\Vert 
\mathbf{v}_{0}\right\Vert _{L^{2}\left( 0,T;V\right) }^{2}+\left\Vert 
\mathbf{v}_{1}\right\Vert _{L^{2}\left( Q;L_{per}^{2}\left( Z;V_{Y}\right)
\right) }^{2}\right) ^{\frac{1}{2}}\text{, }\mathbf{v=}\left( \mathbf{v}_{0},%
\mathbf{v}_{1}\right) \in \mathbb{F}_{0}^{1}\text{.}
\end{equation*}%
On the other hand, put 
\begin{equation*}
\mathbf{\tciFourier }_{0}^{\infty }=\mathcal{D}\left( 0,T;\mathcal{V}\right)
\times \left[ \mathcal{D}\left( Q;\mathbb{R}\right) \otimes \left[ \mathcal{C%
}_{per}^{\infty }\left( Z;\mathbb{R}\right) \otimes \mathcal{V}_{Y}\right] %
\right] \text{,}
\end{equation*}%
where $\mathcal{C}_{per}^{\infty }\left( Z;\mathbb{R}\right) =\mathcal{C}%
^{\infty }\left( \mathbb{R};\mathbb{R}\right) \cap \mathcal{C}_{per}\left(
Z\right) $, $\mathcal{C}_{per}^{\infty }\left( Z;\mathbb{R}\right) \otimes 
\mathcal{V}_{Y}$ stands for the space of vector functions $\mathbf{w}$ on $%
\mathbb{R}_{y}^{N}\times \mathbb{R}_{\tau }$ of the form%
\begin{equation*}
\mathbf{w}\left( y,\tau \right) =\sum_{finite}\chi _{i}\left( \tau \right) 
\mathbf{v}_{i}\left( y\right) \text{ \ }\left( \tau \in \mathbb{R},\text{ }%
y\in \mathbb{R}^{N}\right)
\end{equation*}%
with $\chi _{i}\in \mathcal{C}_{per}^{\infty }\left( Z;\mathbb{R}\right) $, $%
\mathbf{v}_{i}\in \mathcal{V}_{Y}$, and where $\mathcal{D}\left( Q;\mathbb{R}%
\right) \otimes \left( \mathcal{C}_{per}^{\infty }\left( Z;\mathbb{R}\right)
\otimes \mathcal{V}_{Y}\right) $ is the space of vector functions on $%
Q\times \mathbb{R}_{y}^{N}\times \mathbb{R}$ of the form 
\begin{equation*}
\mathbf{\psi }\left( x,t,y,\tau \right) =\sum_{finite}\varphi _{i}\left(
x,t\right) \mathbf{w}_{i}\left( y,\tau \right) \text{ }\left( \left(
x,t\right) \in Q,\text{ }\left( y,\tau \right) \in \mathbb{R}^{N}\times 
\mathbb{R}\right)
\end{equation*}%
with $\varphi _{i}\in \mathcal{D}\left( Q;\mathbb{R}\right) $, $\mathbf{w}%
_{i}\in \mathcal{C}_{per}^{\infty }\left( Z;\mathbb{R}\right) \otimes 
\mathcal{V}_{Y}$. Since $\mathcal{V}$ is dense in $V$ (see \cite[p.18]{bib18}%
), it is clear that $\mathbf{\tciFourier }_{0}^{\infty }$ is dense in $%
\mathbb{F}_{0}^{1}$.

Now, let 
\begin{equation*}
\mathbb{U}=V\times L^{2}\left( \Omega ;L_{per}^{2}\left( Z;V_{Y}\right)
\right) \text{.}
\end{equation*}%
Provided with the norm%
\begin{equation*}
\left\Vert \mathbf{v}\right\Vert _{\mathbb{U}}=\left( \left\Vert \mathbf{v}%
_{0}\right\Vert ^{2}+\left\Vert \mathbf{v}_{1}\right\Vert _{L^{2}\left(
\Omega ;L_{per}^{2}\left( Z;V_{Y}\right) \right) }^{2}\right) ^{\frac{1}{2}%
}\qquad \left( \mathbf{v}=\left( \mathbf{v}_{0},\mathbf{v}_{1}\right) \in 
\mathbb{U}\right) \text{,}
\end{equation*}%
$\mathbb{U}$ is a Hilbert space. Let us set%
\begin{equation*}
\widehat{a}_{\Omega }\left( \mathbf{u},\mathbf{v}\right)
=\sum_{i,j,k=1}^{N}\int \int \int_{\Omega \times Y\times Z}a_{ij}\left( 
\frac{\partial u_{0}^{k}}{\partial x_{j}}+\frac{\partial u_{1}^{k}}{\partial
y_{j}}\right) \left( \frac{\partial v_{0}^{k}}{\partial x_{i}}+\frac{%
\partial v_{1}^{k}}{\partial y_{i}}\right) dxdyd\tau
\end{equation*}%
for $\mathbf{u=}\left( \mathbf{u}_{0},\mathbf{u}_{1}\right) $ and $\mathbf{v=%
}\left( \mathbf{v}_{0},\mathbf{v}_{1}\right) $ in $\mathbb{U}$. This defines
a symmetric continuous bilinear form $\widehat{a}_{\Omega }$ on $\mathbb{U}%
\times \mathbb{U}$. Furthermore, $\widehat{a}_{\Omega }$ is $\mathbb{U}$%
-elliptic. Specifically, 
\begin{equation}
\widehat{a}_{\Omega }\left( \mathbf{u},\mathbf{u}\right) \geq \alpha
\left\Vert \mathbf{u}\right\Vert _{\mathbb{U}}^{2}\text{ }\left( \mathbf{u}%
\in \mathbb{U}\right)  \label{eq3.6}
\end{equation}%
as is easily checked by using (\ref{eq1.2}) and the fact that $\int_{Y}\frac{%
\partial u_{1}^{k}}{\partial y_{j}}\left( x,y,\tau \right) dy=0$.

Here is one fundamental lemma.

\begin{lemma}
\label{lem3.1} Under the hypotheses (\ref{eq1.1})-(\ref{eq1.2}). The
variational problem 
\begin{equation}
\left\{ 
\begin{array}{c}
\mathbf{u}_{0}\in \mathcal{W}\left( 0,T\right) \text{ with }\mathbf{u}%
_{0}\left( 0\right) =0\text{;}\qquad \qquad \qquad \qquad \qquad \qquad \quad
\\ 
\mathbf{u}=\left( \mathbf{u}_{0},\mathbf{u}_{1}\right) \in \mathbb{F}%
_{0}^{1}:\qquad \qquad \qquad \qquad \qquad \qquad \qquad \qquad \quad \\ 
\int_{0}^{T}\left( \mathbf{u}_{0}^{\prime }\left( t\right) ,\mathbf{v}%
_{0}\left( t\right) \right) dt+\int_{0}^{T}\widehat{a}_{\Omega }\left( 
\mathbf{u}\left( t\right) ,\mathbf{v}\left( t\right) \right)
dt=\int_{0}^{T}\left( \mathbf{f}\left( t\right) ,\mathbf{v}_{0}\left(
t\right) \right) dt \\ 
\text{for all }\mathbf{v=}\left( \mathbf{v}_{0},\mathbf{v}_{1}\right) \in 
\mathbb{F}_{0}^{1}%
\end{array}%
\right.  \label{eq3.2}
\end{equation}%
has at most one solution.
\end{lemma}

\begin{proof}
Let $\mathbf{v}_{\ast }=\left( \mathbf{v}_{0},\mathbf{v}_{1}\right) \in 
\mathbb{U}$ and $\varphi \in \mathcal{D}\left( ]0,T[\right) $. By taking $%
\mathbf{v=}\varphi \otimes \mathbf{v}_{\ast }$ in (\ref{eq3.2}), we arrive
at 
\begin{equation}
\left( \mathbf{u}_{0}^{\prime }\left( t\right) ,\mathbf{v}_{0}\right) +%
\widehat{a}_{\Omega }\left( \mathbf{u}\left( t\right) ,\mathbf{v}_{\ast
}\right) =\left( \mathbf{f}\left( t\right) ,\mathbf{v}_{0}\right) \text{%
\qquad }\left( \mathbf{v}_{\ast }\in \mathbb{U}\right)  \label{eq3.3}
\end{equation}%
for almost all $t\in \left( 0,T\right) $. Suppose that $\mathbf{u}_{\ast }$
and\textbf{\ }$\mathbf{u}_{\ast \ast }$ are two solutions of (\ref{eq3.2})
with $\mathbf{u}_{\ast }=\left( \mathbf{u}_{\ast 0},\mathbf{u}_{\ast
1}\right) $ and $\mathbf{u}_{\ast \ast }=\left( \mathbf{u}_{\ast \ast 0},%
\mathbf{u}_{\ast \ast 1}\right) $. Let $\mathbf{u=u}_{\ast }-\mathbf{u}%
_{\ast \ast }=\left( \mathbf{u}_{0},\mathbf{u}_{1}\right) $ with $\mathbf{u}%
_{0}=\mathbf{u}_{\ast 0}-\mathbf{u}_{\ast \ast 0}$ and $\mathbf{u}_{1}=%
\mathbf{u}_{\ast 1}-\mathbf{u}_{\ast \ast 1}$. Let us show that $\mathbf{u=}%
0 $. By\ using (\ref{eq3.3}) we see that $\mathbf{u}$ verifies: 
\begin{equation}
\left( \mathbf{u}_{0}^{\prime }\left( t\right) ,\mathbf{v}_{0}\right) +%
\widehat{a}_{\Omega }\left( \mathbf{u}\left( t\right) ,\mathbf{v}_{\ast
}\right) =0  \label{eq3.4}
\end{equation}%
for all $\mathbf{v}_{\ast }\in \mathbb{U}$ and for almost all $t\in \left(
0,T\right) $. But, by virtue of \cite[p. 261]{bib18} 
\begin{equation*}
\frac{d}{dt}\left\vert \mathbf{u}_{0}\left( t\right) \right\vert
^{2}=2\left( \mathbf{u}_{0}^{\prime }\left( t\right) ,\mathbf{u}_{0}\left(
t\right) \right) \text{\qquad }
\end{equation*}%
for almost all $t\in \left( 0,T\right) $. Then, taking $\mathbf{v}_{\ast }%
\mathbf{=u}\left( t\right) $ in (\ref{eq3.4}), we obtain by\ (\ref{eq3.6}) 
\begin{equation}
\frac{d}{dt}\left\vert \mathbf{u}_{0}\left( t\right) \right\vert
^{2}+2\alpha \left\Vert \mathbf{u}\left( t\right) \right\Vert _{\mathbb{U}%
}^{2}\leq 0\text{\qquad }  \label{eq3.5}
\end{equation}%
for almost all $t\in \left( 0,T\right) $. Integrating (\ref{eq3.5}) on $%
\left[ 0,t\right] $ $\left( 0\leq t\leq T\right) $, we get $\left\vert 
\mathbf{u}_{0}\left( t\right) \right\vert ^{2}\leq 0$ for all $t\in \left[
0,T\right] $ and $\left\Vert \mathbf{u}\right\Vert _{\mathbb{F}%
_{0}^{1}}^{2}\leq 0$, thus $\mathbf{u}=0$ and the lemma follows.
\end{proof}

We are now able to prove the desired theorem. Throughout the remainder of
the present section, it is assumed that $a_{ij}$ is $Y$-periodic for any $%
1\leq i,j\leq N$.

\begin{theorem}
\label{th3.3} Suppose that the hypotheses of Lemma \ref{lem2.1} are
satisfied. For $0<\varepsilon <1$, let $\mathbf{u}_{\varepsilon }$ be
defined by (\ref{eq1.3})-(\ref{eq1.6}). Then, as $\varepsilon \rightarrow 0$
we have%
\begin{equation}
\mathbf{u}_{\varepsilon }\rightarrow \mathbf{u}_{0}\text{ in }\mathcal{W}%
\left( 0,T\right) \text{-weak,\qquad \qquad \qquad \qquad \qquad \qquad }
\label{eq3.7}
\end{equation}%
\begin{equation}
\frac{\partial u_{\varepsilon }^{k}}{\partial x_{j}}\rightarrow \frac{%
\partial u_{0}^{k}}{\partial x_{j}}+\frac{\partial u_{1}^{k}}{\partial y_{j}}%
\text{ in }L^{2}\left( Q\right) \text{-weak }2\text{-}s\text{ }\left( 1\leq
j,k\leq N\right)  \label{eq3.8}
\end{equation}%
where $\mathbf{u=}\left( \mathbf{u}_{0},\mathbf{u}_{1}\right) $ (with $%
\mathbf{u}_{0}=\left( u_{0}^{k}\right) $ and $\mathbf{u}_{1}=\left(
u_{1}^{k}\right) $) is the unique solution of (\ref{eq3.2}).
\end{theorem}

\begin{proof}
By Proposition \ref{pr2.2}, we see that the sequences $\left( p_{\varepsilon
}\right) _{0<\varepsilon <1}$ and $\left( \mathbf{u}_{\varepsilon }\right)
_{0<\varepsilon <1}=\left( u_{\varepsilon }^{1},...,u_{\varepsilon
}^{N}\right) _{0<\varepsilon <1}$ are bounded respectively in $L^{2}\left(
Q\right) $ and $\mathcal{W}\left( 0,T\right) $. Further, it follows from (%
\ref{eq2.5}) and (\ref{eq2.6}) that for $1\leq k\leq N$, the sequence $%
\left( u_{\varepsilon }^{k}\right) _{0<\varepsilon <1}$ is bounded in $%
\mathcal{Y}\left( 0,T\right) $. Let $E$ be a fundamental sequence. Then, by
Theorems \ref{th3.1}-\ref{th3.2} and the fact that $\mathcal{W}\left(
0,T\right) $ is compactly embedded in $L^{2}\left( Q\right) ^{N}$, there
exist a subsequence $E^{\prime }$ extracted from $E$ and functions $\mathbf{u%
}_{0}=\left( u_{0}^{k}\right) _{1\leq k\leq N}\in \mathcal{W}\left(
0,T\right) $, $\mathbf{u}_{1}=\left( u_{1}^{k}\right) _{1\leq k\leq N}\in
L^{2}\left( Q;L_{per}^{2}\left( Z;H_{\#}^{1}\left( Y;\mathbb{R}\right)
^{N}\right) \right) $, and $p\in L^{2}\left( Q;L_{per}^{2}\left( Y\times Z;%
\mathbb{R}\right) \right) $ such that as $E^{\prime }\ni \varepsilon
\rightarrow 0$, we have (\ref{eq3.7})-(\ref{eq3.8}) and 
\begin{equation}
\mathbf{u}_{\varepsilon }\rightarrow \mathbf{u}_{0}\text{ in }L^{2}\left(
Q\right) ^{N}\text{-strong,}  \label{eq3.9}
\end{equation}%
\begin{equation}
p_{\varepsilon }\rightarrow p\text{ in }L^{2}\left( Q\right) \text{-weak }2%
\text{-}s\text{.}  \label{eq3.10}
\end{equation}%
But, by virtue of Lemma \ref{lem3.1}, the theorem will be entirely proved if
we show that $\mathbf{u=}\left( \mathbf{u}_{0},\mathbf{u}_{1}\right) $
verifies (\ref{eq3.2}). In fact, according to (\ref{eq1.4}), we have $div%
\mathbf{u}_{0}=0$ and $div_{y}\mathbf{u}_{1}=0$. Therefore $\mathbf{u=}%
\left( \mathbf{u}_{0},\mathbf{u}_{1}\right) \in \mathbb{F}_{0}^{1}$. Let us
recall that $\mathbf{u}_{0}$ can be considered as a continuous function of $%
\left[ 0,T\right] $ into $H$ since $\mathcal{W}\left( 0,T\right) $ is
continuously embedded in $\mathcal{C}\left( \left[ 0,T\right] ;H\right) $.
Let us show that $\mathbf{u}_{0}\left( 0\right) =0$. For $\mathbf{v}\in V$
and $\varphi \in \mathcal{C}^{1}\left( \left[ 0,T\right] \right) $ with $%
\varphi \left( T\right) =0$ and $\varphi \left( 0\right) =1$, we have by an
integration by part 
\begin{equation*}
\int_{0}^{T}\left( \mathbf{u}_{\varepsilon }^{\prime }\left( t\right) ,%
\mathbf{v}\right) \varphi \left( t\right) dt+\int_{0}^{T}\left( \mathbf{u}%
_{\varepsilon }\left( t\right) ,\mathbf{v}\right) \varphi ^{\prime }\left(
t\right) dt=-\left( \mathbf{u}_{\varepsilon }\left( 0\right) ,\mathbf{v}%
\right) \text{.}
\end{equation*}%
According to (\ref{eq1.6}), we have by passing to the limit in the preceding
equality as $E^{\prime }\ni \varepsilon \rightarrow 0$%
\begin{equation*}
\int_{0}^{T}\left( \mathbf{u}_{0}^{\prime }\left( t\right) ,\mathbf{v}%
\right) \varphi \left( t\right) dt+\int_{0}^{T}\left( \mathbf{u}_{0}\left(
t\right) ,\mathbf{v}\right) \varphi ^{\prime }\left( t\right) dt=0\text{.}
\end{equation*}%
Hence $\left( \mathbf{u}_{0}\left( 0\right) ,\mathbf{v}\right) =0$ for all $%
\mathbf{v}\in V$, and as $V$ is dense in $H$ we conclude that $\mathbf{u}%
_{0}\left( 0\right) =0$. Now, let us check that $\mathbf{u=}\left( \mathbf{u}%
_{0},\mathbf{u}_{1}\right) $ verifies the variational equation of (\ref%
{eq3.2}). For $0<\varepsilon <1$, let%
\begin{equation}
\begin{array}{c}
\mathbf{\Phi }_{\varepsilon }=\mathbf{\psi }_{0}+\varepsilon \mathbf{\psi }%
_{1}^{\varepsilon }\text{ with }\mathbf{\psi }_{0}\in \mathcal{D}\left( Q;%
\mathbb{R}\right) ^{N}\text{ and } \\ 
\mathbf{\psi }_{1}\in \mathcal{D}\left( Q;\mathbb{R}\right) \otimes \left[ 
\mathcal{C}_{per}^{\infty }\left( Z;\mathbb{R}\right) \otimes \mathcal{V}_{Y}%
\right] \text{,}%
\end{array}
\label{eq3.11}
\end{equation}%
i.e., $\mathbf{\Phi }_{\varepsilon }\left( x,t\right) =\mathbf{\psi }%
_{0}\left( x,t\right) +\varepsilon \mathbf{\psi }_{1}\left( x,t,\frac{x}{%
\varepsilon },\frac{t}{\varepsilon }\right) $ for $\left( x,t\right) \in Q$.
We have $\mathbf{\Phi }_{\varepsilon }\in \mathcal{D}\left( Q;\mathbb{R}%
\right) ^{N}$. Thus, multiplying (\ref{eq1.3}) by $\mathbf{\Phi }%
_{\varepsilon }$ yields 
\begin{equation}
\begin{array}{c}
\int_{0}^{T}\left( \mathbf{u}_{\varepsilon }^{\prime }\left( t\right) ,%
\mathbf{\Phi }_{\varepsilon }\left( t\right) \right)
dt+\int_{0}^{T}a^{\varepsilon }\left( \mathbf{u}_{\varepsilon }\left(
t\right) ,\mathbf{\Phi }_{\varepsilon }\left( t\right) \right) dt \\ 
-\int_{Q}p_{\varepsilon }div\mathbf{\Phi }_{\varepsilon
}dxdt=\int_{0}^{T}\left( \mathbf{f}\left( t\right) ,\mathbf{\Phi }%
_{\varepsilon }\left( t\right) \right) dt\text{.}%
\end{array}
\label{eq3.12}
\end{equation}%
Let us note at once that 
\begin{equation*}
\int_{0}^{T}\left( \mathbf{u}_{\varepsilon }^{\prime }\left( t\right) ,%
\mathbf{\Phi }_{\varepsilon }\left( t\right) \right)
dt=-\sum_{l=1}^{N}\int_{Q}u_{\varepsilon }^{l}\left[ \frac{\partial \psi
_{0}^{l}}{\partial t}+\varepsilon \left( \frac{\partial \psi _{1}^{l}}{%
\partial t}\right) ^{\varepsilon }+\left( \frac{\partial \psi _{1}^{l}}{%
\partial \tau }\right) ^{\varepsilon }\right] dxdt\text{.}
\end{equation*}%
Then by virtue of (\ref{eq3.9}) we have%
\begin{equation}
\int_{0}^{T}\left( \mathbf{u}_{\varepsilon }^{\prime }\left( t\right) ,%
\mathbf{\Phi }_{\varepsilon }\left( t\right) \right) dt\rightarrow
-\sum_{l=1}^{N}\int_{Q}u_{0}^{l}\frac{\partial \psi _{0}^{l}}{\partial t}%
dxdt=\int_{0}^{T}\left( \mathbf{u}_{0}^{\prime }\left( t\right) ,\mathbf{%
\psi }_{0}\left( t\right) \right) dt  \label{eq3.13}
\end{equation}%
as $E^{\prime }\ni \varepsilon \rightarrow 0$. In fact, on one hand 
\begin{equation*}
\sum_{l=1}^{N}\int_{Q}u_{\varepsilon }^{l}\left[ \frac{\partial \psi _{0}^{l}%
}{\partial t}+\varepsilon \left( \frac{\partial \psi _{1}^{l}}{\partial t}%
\right) ^{\varepsilon }+\left( \frac{\partial \psi _{1}^{l}}{\partial \tau }%
\right) ^{\varepsilon }\right] dxdt
\end{equation*}%
\begin{equation*}
\rightarrow \sum_{l=1}^{N}\left[ \int_{Q}u_{0}^{l}\frac{\partial \psi
_{0}^{l}}{\partial t}dxdt+\int \int \int_{Q\times Y\times Z}u_{0}^{l}\frac{%
\partial \psi _{1}^{l}}{\partial \tau }dxdtdyd\tau \right]
\end{equation*}%
as $E^{\prime }\ni \varepsilon \rightarrow 0$, on the other hand

$\int \int \int_{Q\times Y\times Z}u_{0}^{l}\frac{\partial \psi _{1}^{l}}{%
\partial \tau }dxdtdyd\tau =\int_{Q}u_{0}^{l}\left( \int \int_{Y\times Z}%
\frac{\partial \psi _{1}^{l}}{\partial \tau }dyd\tau \right) dxdt=0$ by
virtue of the $Y\times Z$-periodicity. The next point is to pass to the
limit in (\ref{eq3.12}) as $E^{\prime }\ni \varepsilon \rightarrow 0$. To
this end, we note that as $E^{\prime }\ni \varepsilon \rightarrow 0$, 
\begin{equation*}
\int_{0}^{T}a^{\varepsilon }\left( \mathbf{u}_{\varepsilon }\left( t\right) ,%
\mathbf{\Phi }_{\varepsilon }\left( t\right) \right) dt\rightarrow
\int_{0}^{T}\widehat{a}_{\Omega }\left( \mathbf{u}\left( t\right) ,\mathbf{%
\Phi }\left( t\right) \right) dt\text{,}
\end{equation*}%
where $\mathbf{\Phi }=\left( \mathbf{\psi }_{0},\mathbf{\psi }_{1}\right) $
(proceed as in the proof of the analogous result in \cite[p.179]{bib13}).
Now, based on (\ref{eq3.10}), there is no difficulty in showing that as $%
E^{\prime }\ni \varepsilon \rightarrow 0$,%
\begin{equation*}
\int_{Q}p_{\varepsilon }div\mathbf{\Phi }_{\varepsilon }dxdt\rightarrow \int
\int \int_{Q\times Y\times Z}pdiv\mathbf{\psi }_{0}dxdtdyd\tau \text{.}
\end{equation*}%
On the other hand, let us check that as $\varepsilon \rightarrow 0$%
\begin{equation}
\int_{0}^{T}\left( \mathbf{f}\left( t\right) ,\mathbf{\Phi }_{\varepsilon
}\left( t\right) \right) dt\rightarrow \int_{0}^{T}\left( \mathbf{f}\left(
t\right) ,\mathbf{\psi }_{0}\left( t\right) \right) dt\text{.}
\label{eq3.14}
\end{equation}%
Indeed, if $\mathbf{f}\in L^{2}\left( 0,T;L^{2}\left( \Omega ;\mathbb{R}%
\right) ^{N}\right) $ (\ref{eq3.14}) is immediate by using the classical
fact that $\mathbf{\Phi }_{\varepsilon }\rightarrow \mathbf{\psi }_{0}$ in $%
L^{2}\left( Q\right) ^{N}$-weak and $\frac{\partial \mathbf{\Phi }%
_{\varepsilon }}{\partial x_{j}}\rightarrow \frac{\partial \mathbf{\psi }_{0}%
}{\partial x_{j}}$ in $L^{2}\left( Q\right) ^{N}$-weak $\left( 1\leq j\leq
N\right) $ as $\varepsilon \rightarrow 0$. In the general case, (\ref{eq3.14}%
) follows by the density of $L^{2}\left( 0,T;L^{2}\left( \Omega ;\mathbb{R}%
\right) ^{N}\right) $ in $L^{2}\left( 0,T;H^{-1}\left( \Omega ;\mathbb{R}%
\right) ^{N}\right) $.

Having made this point, we can pass to the limit in (\ref{eq3.12}) when $%
E^{\prime }\ni \varepsilon \rightarrow 0$, and the result is that%
\begin{equation}
\begin{array}{c}
\int_{0}^{T}\left( \mathbf{u}_{0}^{\prime }\left( t\right) ,\mathbf{\psi }%
_{0}\left( t\right) \right) dt+\int_{0}^{T}\widehat{a}_{\Omega }\left( 
\mathbf{u}\left( t\right) ,\mathbf{\Phi }\left( t\right) \right) dt \\ 
-\int_{Q}p_{0}div\mathbf{\psi }_{0}dxdt=\int_{0}^{T}\left( \mathbf{f}\left(
t\right) ,\mathbf{\psi }_{0}\left( t\right) \right) dt\text{,}%
\end{array}
\label{eq3.15}
\end{equation}%
where $p_{0}$ denotes the mean of $p$, i.e., $p_{0}\in L^{2}\left(
0,T;L^{2}\left( \Omega ;\mathbb{R}\right) \right) $ and $p_{0}\left(
x,t\right) =\int \int_{Y\times Z}p\left( x,t,y,\tau \right) dyd\tau $ a.e.
in $\left( x,t\right) \in Q$, and where $\mathbf{\Phi }=\left( \mathbf{\psi }%
_{0},\mathbf{\psi }_{1}\right) $, $\mathbf{\psi }_{0}$ ranging over $%
\mathcal{D}\left( Q;\mathbb{R}\right) ^{N}$ and $\mathbf{\psi }_{1}$ ranging
over $\mathcal{D}\left( Q;\mathbb{R}\right) \otimes \left[ \mathcal{C}%
_{per}^{\infty }\left( Z;\mathbb{R}\right) \otimes \mathcal{V}_{Y}\right] $.
Taking in particular $\mathbf{\psi }_{0}$ in $\mathcal{D}\left( 0,T;\mathcal{%
V}\right) $ and using the density of $\mathbf{\tciFourier }_{0}^{\infty }$
in $\mathbb{F}_{0}^{1}$, one quickly arrives at (\ref{eq3.2}). The unicity
of $\mathbf{u=}\left( \mathbf{u}_{0},\mathbf{u}_{1}\right) $ follows by
Lemma \ref{lem3.1}. Consequently, (\ref{eq3.7}) and (\ref{eq3.8}) still hold
when $E\ni \varepsilon \rightarrow 0$. Hence when $0<\varepsilon \rightarrow
0$, by virtue of the arbitrariness of $E$. The theorem is proved.
\end{proof}

Now, we wish to give a simple representation of the vector function $\mathbf{%
u}_{1}$ in Theorem \ref{th3.3} for further uses. For this purpose we
introduce the bilinear form $\widehat{a}$ on $L_{per}^{2}\left(
Z;V_{Y}\right) \times L_{per}^{2}\left( Z;V_{Y}\right) $ defined by%
\begin{equation*}
\widehat{a}\left( \mathbf{u},\mathbf{v}\right) =\sum_{i,j,k=1}^{N}\int
\int_{Y\times Z}a_{ij}\frac{\partial u^{k}}{\partial y_{j}}\frac{\partial
v^{k}}{\partial y_{i}}dyd\tau
\end{equation*}%
for $\mathbf{u=}\left( u^{k}\right) $ and $\mathbf{v=}\left( v^{k}\right)
\in L_{per}^{2}\left( Z;V_{Y}\right) $. Next, for each pair of indices $%
1\leq i,k\leq N$, we consider the variational problem%
\begin{equation}
\left\{ 
\begin{array}{c}
\mathbf{\chi }_{ik}\in L_{per}^{2}\left( Z;V_{Y}\right) :\qquad \qquad \qquad
\\ 
\widehat{a}\left( \mathbf{\chi }_{ik},\mathbf{w}\right)
=\sum_{l=1}^{N}\int_{Y\times Z}a_{li}\frac{\partial w^{k}}{\partial y_{l}}%
dyd\tau \\ 
\text{for all }\mathbf{w}=\left( w^{j}\right) \in L_{per}^{2}\left(
Z;V_{Y}\right) \text{,}%
\end{array}%
\right.  \label{eq3.16}
\end{equation}%
which determines $\mathbf{\chi }_{ik}$ in a unique manner.

\begin{lemma}
\label{lem3.2} Under the hypotheses and notation of Theorem \ref{th3.3}, we
have%
\begin{equation}
\mathbf{u}_{1}\left( x,t,y,\tau \right) =-\sum_{i,k=1}^{N}\frac{\partial
u_{0}^{k}}{\partial x_{i}}\left( x,t\right) \mathbf{\chi }_{ik}\left( y,\tau
\right)  \label{eq3.17}
\end{equation}%
almost everywhere in $\left( x,t,y,\tau \right) \in Q\times Y\times Z$.
\end{lemma}

\begin{proof}
In (\ref{eq3.2}), we choose the test functions $\mathbf{v}=\left( \mathbf{v}%
_{0},\mathbf{v}_{1}\right) $ such that $\mathbf{v}_{0}=0$ and $\mathbf{v}%
_{1}\left( x,t,y,\tau \right) =\varphi \left( x,t\right) \mathbf{w}\left(
y,\tau \right) $ for $\left( x,t,y,\tau \right) \in Q\times Y\times Z$,
where $\varphi \in \mathcal{D}\left( Q;\mathbb{R}\right) $ and $\mathbf{w}%
\in L_{per}^{2}\left( Z;V_{Y}\right) $. Then for almost every $\left(
x,t\right) $ in $Q$, we have%
\begin{equation}
\left\{ 
\begin{array}{c}
\widehat{a}\left( \mathbf{u}_{1}\left( x,t\right) ,\mathbf{w}\right)
=-\sum_{l,j,k=1}^{N}\frac{\partial u_{0}^{k}}{\partial x_{j}}\left(
x,t\right) \int \int_{Y\times Z}a_{lj}\frac{\partial w^{k}}{\partial y_{l}}%
dyd\tau \\ 
\text{for all }\mathbf{w}\in L_{per}^{2}\left( Z;V_{Y}\right) \text{.}\qquad
\qquad \qquad \qquad \qquad \qquad%
\end{array}%
\right.  \label{eq3.18}
\end{equation}%
But it is clear that $\mathbf{u}_{1}\left( x,t\right) $ (for fixed $\left(
x,t\right) \in Q$) is the unique function in $L_{per}^{2}\left(
Z;V_{Y}\right) $ solving the variational equation (\ref{eq3.18}). On the
other hand, it is an easy exercise to verify that $\mathbf{z}\left(
x,t\right) =-\sum_{i,k=1}^{N}\frac{\partial u_{0}^{k}}{\partial x_{i}}\left(
x,t\right) \mathbf{\chi }_{ik}$ solves also (\ref{eq3.18}). Hence the lemma
follows immediately.
\end{proof}

\subsection{Macroscopic homogenized equations}

Our aim here is to derive a well-posed initial boundary value problem for $%
\left( \mathbf{u}_{0},p_{0}\right) $. To begin, for $1\leq i,j,k,h\leq N$,
let%
\begin{equation*}
q_{ijkh}=\delta _{kh}\int_{Y}a_{ij}\left( y\right) dy-\sum_{l=1}^{N}\int
\int_{Y\times Z}a_{il}\left( y\right) \frac{\partial \mathcal{\chi }_{jh}^{k}%
}{\partial y_{l}}\left( y,\tau \right) dyd\tau \text{,}
\end{equation*}%
where: $\delta _{kh}$ is the Kronecker symbol, $\mathbf{\chi }_{jh}=\left( 
\mathcal{\chi }_{jh}^{k}\right) $ is defined by (\ref{eq3.16}). To the
coefficients $q_{ijkh}$ we associate the differential operator $\mathcal{Q}$
on $Q$ mapping $\mathcal{D}^{\prime }\left( Q\right) ^{N}$ into $\mathcal{D}%
^{\prime }\left( Q\right) ^{N}$ ($\mathcal{D}^{\prime }\left( Q\right) $
being the usual space of complex distributions on $Q$) as 
\begin{equation}
(\mathcal{Q}\mathbf{z})^{k}=-\sum_{i,j,h=1}^{N}q_{ijkh}\frac{\partial
^{2}z^{h}}{\partial x_{i}\partial x_{j}}\quad \left( 1\leq k\leq N\right) 
\text{ for }\mathbf{z}=\left( z^{h}\right) \text{, }z^{h}\in \mathcal{D}%
^{\prime }\left( Q\right) \text{.}  \label{eq3.19}
\end{equation}%
$\mathcal{Q}$ is the so-called homogenized operator associated to $%
P^{\varepsilon }$ $\left( 0<\varepsilon <1\right) $.

Now, let us consider the initial boundary value\ problem 
\begin{equation}
\frac{\partial \mathbf{u}_{0}}{\partial t}+\mathcal{Q}\mathbf{u}_{0}+\mathbf{%
grad}p_{0}=\mathbf{f}\text{ in }Q=\Omega \times ]0,T[\text{,}  \label{eq3.20}
\end{equation}%
\begin{equation}
div\mathbf{u}_{0}=0\text{ in }Q\text{,}  \label{eq3.21}
\end{equation}%
\begin{equation}
\mathbf{u}_{0}=0\text{ on }\partial \Omega \times ]0,T[\text{,}
\label{eq3.22}
\end{equation}%
\begin{equation}
\mathbf{u}_{0}\left( 0\right) =0\text{ in }\Omega \text{.}  \label{eq3.23}
\end{equation}

\begin{lemma}
\label{lem3.3} The initial boundary value problem (\ref{eq3.20})-(\ref%
{eq3.23}) admits at most one weak solution $\left( \mathbf{u}%
_{0},p_{0}\right) $ with

$\mathbf{u}_{0}\in \mathcal{W}\left( 0,T\right) $ and $p_{0}\in L^{2}\left(
0,T;L^{2}\left( \Omega ;\mathbb{R}\right) \mathfrak{/}\mathbb{R}\right) $.
\end{lemma}

\begin{proof}
If $\left( \mathbf{u}_{0},p_{0}\right) \in \mathcal{W}\left( 0,T\right)
\times L^{2}\left( 0,T;L^{2}\left( \Omega ;\mathbb{R}\right) \right) $
verifies (\ref{eq3.20})-(\ref{eq3.23}), then we have%
\begin{equation*}
\begin{array}{c}
\int_{0}^{T}\left( \mathbf{u}_{0}^{\prime }\left( t\right) ,\mathbf{v}%
_{0}\left( t\right) \right) dt+\sum_{i,j,k,h=1}^{N}\int_{Q}q_{ijkh}\frac{%
\partial u_{0}^{h}}{\partial x_{j}}\frac{\partial v_{0}^{k}}{\partial x_{i}}%
dxdt \\ 
=\int_{0}^{T}\left( \mathbf{f}\left( t\right) ,\mathbf{v}_{0}\left( t\right)
\right) dt%
\end{array}%
\end{equation*}%
for all $\mathbf{v}_{0}\in L^{2}\left( 0,T;V\right) $. From the previous
equality, one quickly arrives at 
\begin{equation}
\begin{array}{c}
\int_{0}^{T}\left( \mathbf{u}_{0}^{\prime }\left( t\right) ,\mathbf{v}%
_{0}\left( t\right) \right) dt+\sum_{i,j,k=1}^{N}\int \int \int_{Q\times
Y\times Z}a_{ij}\left( \frac{\partial u_{0}^{k}}{\partial x_{j}}+\frac{%
\partial u_{1}^{k}}{\partial y_{j}}\right) \frac{\partial v_{0}^{k}}{%
\partial x_{i}}dxdtdyd\tau \\ 
=\int_{0}^{T}\left( \mathbf{f}\left( t\right) ,\mathbf{v}_{0}\left( t\right)
\right) dt%
\end{array}
\label{eq3.24}
\end{equation}%
where $u_{1}^{k}\left( x,t,y,\tau \right) =-\sum_{i,h=1}^{N}\frac{\partial
u_{0}^{h}}{\partial x_{i}}\left( x,t\right) \mathcal{\chi }_{ih}^{k}\left(
y,\tau \right) $ for $\left( x,t,y,\tau \right) \in Q\times Y\times Z$. Let
us check that $\mathbf{u}=\left( \mathbf{u}_{0},\mathbf{u}_{1}\right) $
(with $\mathbf{u}_{1}\left( x,t,y,\tau \right) =-\sum_{i,k=1}^{N}\frac{%
\partial u_{0}^{k}}{\partial x_{i}}\left( x,t\right) \mathbf{\chi }%
_{ik}\left( y,\tau \right) $ for $\left( x,t,y,\tau \right) \in Q\times
Y\times Z$) satisfies (\ref{eq3.2}). Indeed, we have 
\begin{equation}
\sum_{i,j,k=1}^{N}\int \int \int_{Q\times Y\times Z}a_{ij}\left( \frac{%
\partial u_{0}^{k}}{\partial x_{j}}+\frac{\partial u_{1}^{k}}{\partial y_{j}}%
\right) \frac{\partial v_{1}^{k}}{\partial y_{i}}dxdtdyd\tau =0
\label{eq3.25}
\end{equation}%
for all $\mathbf{v}_{1}=\left( v_{1}^{k}\right) \in L^{2}\left(
Q;L_{per}^{2}\left( Z;V_{Y}\right) \right) $, since $\mathbf{u}_{1}\left(
x,t\right) $ verifies (\ref{eq3.18}) for $\left( x,t\right) \in Q$. Thus, by
(\ref{eq3.24})-(\ref{eq3.25}), we see that $\mathbf{u}=\left( \mathbf{u}_{0},%
\mathbf{u}_{1}\right) $ verifies (\ref{eq3.2}). Hence, the unicity in (\ref%
{eq3.20})-(\ref{eq3.23}) follows by Lemme \ref{lem3.1}.
\end{proof}

This leads us to the following theorem.

\begin{theorem}
\label{th3.4} Suppose that the hypotheses of Theorem \ref{th3.3} are
satisfied. For each $0<\varepsilon <1$, let $\left( \mathbf{u}_{\varepsilon
},p_{\varepsilon }\right) \in \mathcal{W}\left( 0,T\right) \times
L^{2}\left( 0,T;L^{2}\left( \Omega ;\mathbb{R}\right) \mathfrak{/}\mathbb{R}%
\right) $ be defined by (\ref{eq1.3})-(\ref{eq1.6}). Then, as $\varepsilon
\rightarrow 0$, we have $\mathbf{u}_{\varepsilon }\rightarrow \mathbf{u}_{0}$
in $\mathcal{W}\left( 0,T\right) $-weak and $p_{\varepsilon }\rightarrow
p_{0}$ in $L^{2}\left( 0,T;L^{2}\left( \Omega \right) \right) $-weak, where
the pair $\left( \mathbf{u}_{0},p_{0}\right) $\ lies in $\mathcal{W}\left(
0,T\right) \times L^{2}\left( 0,T;L^{2}\left( \Omega ;\mathbb{R}\right) 
\mathfrak{/}\mathbb{R}\right) $ and is the unique solution of (\ref{eq3.20}%
)-(\ref{eq3.23}).
\end{theorem}

\begin{proof}
Let $E$ be a fundamental sequence. As in the proof of Theorem \ref{th3.3},
there exists a subsequence $E^{\prime }$ extracted from $E$ such that as $%
E^{\prime }\ni \varepsilon \rightarrow 0$, we have (\ref{eq3.7})-(\ref{eq3.8}%
) and (\ref{eq3.10}) with $\mathbf{u=}\left( \mathbf{u}_{0},\mathbf{u}%
_{1}\right) \in \mathbb{F}_{0}^{1}$ and $\mathbf{u}_{0}\left( 0\right) $.
Then, from (\ref{eq3.10}) we have $p_{\varepsilon }\rightarrow p_{0}$ in $%
L^{2}\left( 0,T;L^{2}\left( \Omega \right) \right) $-weak when $E^{\prime
}\ni \varepsilon \rightarrow 0$, where $p_{0}$ is the mean of $p$. Hence, it
follows that $p_{0}\in L^{2}\left( 0,T;L^{2}\left( \Omega ;\mathbb{R}\right) 
\mathfrak{/}\mathbb{R}\right) $. Furher, (\ref{eq3.15}) holds for all $%
\mathbf{\Phi }=\left( \mathbf{\psi }_{0},\mathbf{\psi }_{1}\right) \in 
\mathcal{D}\left( Q;\mathbb{R}\right) ^{N}\times \mathcal{D}\left( Q;\mathbb{%
R}\right) \otimes \left[ \mathcal{C}_{per}^{\infty }\left( Z;\mathbb{R}%
\right) \otimes \mathcal{V}_{Y}\right] $. Then, substituting (\ref{eq3.17})
in (\ref{eq3.15}) and choosing therein the $\mathbf{\Phi }$'s such that $%
\mathbf{\psi }_{1}=0$, a simple computation leads to (\ref{eq3.20}) with
evidently (\ref{eq3.21})-(\ref{eq3.23}). Hence the Theorem follows by Lemma %
\ref{lem3.3} since $E$ is arbitrarily chosen.
\end{proof}

\begin{remark}
\label{rem3.1} The operator $\mathcal{Q}$ is elliptic, i.e., there is some $%
\alpha _{0}>0$ such that%
\begin{equation*}
\sum_{i,j,k,h=1}^{N}q_{ijkh}\xi _{ik}\xi _{jh}\geq \alpha
_{0}\sum_{k,h=1}^{N}\left\vert \xi _{kh}\right\vert ^{2}
\end{equation*}%
for all $\mathbf{\xi =}\left( \xi _{ij}\right) $ with $\xi _{ij}\in \mathbb{R%
}$. Indeed, by following a classical line of argument (see, e.g., \cite{bib2}%
), we can give a suitable expression of $q_{ijkh}$, viz.%
\begin{equation*}
q_{ijkh}=\widehat{a}\left( \mathbf{\chi }_{ik}-\mathbf{\pi }_{ik},\mathbf{%
\chi }_{jh}-\mathbf{\pi }_{jh}\right) \text{,}
\end{equation*}%
where, for each pair of indices $1\leq i,k\leq N$, the vector function $%
\mathbf{\pi }_{ik}=\left( \pi _{ik}^{1},...,\pi _{ik}^{N}\right) :\mathbb{R}%
_{y}^{N}\rightarrow \mathbb{R}$ is given by $\pi _{ik}^{r}\left( y\right)
=y_{i}\delta _{kr}$ $\left( r=1,...,N\right) $ for $y=\left(
y_{1},...,y_{N}\right) \in \mathbb{R}^{N}$. Hence, the above ellipticity
property follows in a classical fashion.
\end{remark}

\end{document}